\documentclass[a4paper,11pt]{amsart}

\usepackage{amsfonts}
\usepackage{amsmath}
\usepackage{mathrsfs}
\usepackage{amssymb}
\usepackage{amsthm}
\usepackage{amscd}
\usepackage{latexsym}
\usepackage{amstext}
\usepackage{amsxtra}
\usepackage[utf8]{inputenc}
\usepackage[english]{babel}
\usepackage[alphabetic,backrefs,lite]{amsrefs} 

\usepackage{geometry}

\usepackage{color}
\usepackage{paralist}

\usepackage[colorinlistoftodos]{todonotes}
\usepackage[all]{xy}
\usepackage{tikz}

\usepackage{enumerate}
\usepackage{multirow}
\usepackage{wasysym}
\usepackage{latexsym} 
\usepackage{comment}
\usepackage{colonequals}

\usepackage[
        draft=false,
        colorlinks, citecolor=darkgreen,
        backref,       
        pdfauthor={F. F. Favale, G. P. Pirola and S. Torelli},
        pdftitle={Holomorphic 1-forms on the moduli space of curves},
        linktocpage        
]{hyperref}
\hypersetup{citecolor=blue,linktocpage}

\newtheorem{theorem}{Theorem}[section]
\newtheorem*{theorem*}{Theorem}
\newtheorem{lemma}[theorem]{Lemma}
\newtheorem{corollary}[theorem]{Corollary}
\newtheorem*{corollary*}{Corollary}
\newtheorem{proposition}[theorem]{Proposition}
\newtheorem{remark}[theorem]{Remark}
\newtheorem{definition}[theorem]{Definition}

\newtheorem{question}[theorem]{Question}

\usepackage{xpatch}
\makeatletter
\AtBeginDocument{\xpatchcmd{\@thm}{\thm@headpunct{.}}{\thm@headpunct{}}{}{}}
\makeatother

\newcommand{\nc}{\newcommand} 
\nc{\cH}{{\mathcal H}}
\nc{\cA}{{\mathcal A}}
\nc{\cG}{{\mathcal G}}
\nc{\cC}{{\mathcal C}}
\nc{\cD}{{\mathcal D}}
\nc{\cO}{{\mathcal O}}
\nc{\cI}{{\mathcal I}}
\nc{\cB}{{\mathcal B}}
\nc{\cY}{{\mathcal Y}}
\nc{\cK}{{\mathcal K}} 
\nc{\cX}{{\mathcal X}}
\nc{\cS}{{\mathcal S}}
\nc{\cE}{{\mathcal E}}
\nc{\cF}{{\mathcal F}}
\nc{\cZ}{{\mathcal Z}}
\nc{\cQ}{{\mathcal Q}}
\nc{\cN}{{\mathcal N}}
\nc{\cP}{{\mathcal P}}
\nc{\cL}{{\mathcal L}}
\nc{\cM}{{\mathcal M}}
\nc{\cT}{{\mathcal T}}
\nc{\cW}{{\mathcal W}}
\nc{\cU}{{\mathcal U}}
\nc{\cJ}{{\mathcal J}}
\nc{\cV}{{\mathcal V}}
\nc{\bH}{{\mathbb H}}
\nc{\bA}{{\mathbb A}}
\nc{\bG}{{\mathbb G}}
\nc{\bC}{{\mathbb C}}
\nc{\bO}{{\mathbb O}}
\nc{\bI}{{\mathbb I}}
\nc{\bB}{{\mathbb B}}
\nc{\bY}{{\mathbb Y}}
\nc{\bK}{{\mathbb K}} 
\nc{\bX}{{\mathbb X}}
\nc{\bS}{{\mathbb S}}
\nc{\bE}{{\mathbb E}}
\nc{\bF}{{\mathbb F}}
\nc{\bZ}{{\mathbb Z}}
\nc{\bQ}{{\mathbb Q}}
\nc{\bN}{{\mathbb N}}
\nc{\bP}{{\mathbb P}}
\nc{\bL}{{\mathbb L}}
\nc{\bM}{{\mathbb M}}
\nc{\bT}{{\mathbb T}}
\nc{\bW}{{\mathbb W}}
\nc{\bU}{{\mathbb U}}
\nc{\bD}{{\mathbb D}}
\nc{\bJ}{{\mathbb J}}
\nc{\bV}{{\mathbb V}}
\nc{\bbZ}{{\mathbb Z}}
\nc{\bR}{{\mathbb R}}
\nc{\fr}{{\rightarrow}}

\newcommand{\la}{\longrightarrow}
\nc{\cu}{{\barline{\nabla}}}

\nc{\OO}{\mathcal{O}}
\nc{\PP}{\mathbb{P}}
\DeclareMathOperator{\id}{id}

\DeclareMathOperator{\Hom}{Hom}

\DeclareMathOperator{\Pic}{Pic}
\DeclareMathOperator{\NS}{NS}

\DeclareMathOperator{\Sing}{Sing}

\DeclareMathOperator{\Aut}{Aut}

\DeclareMathOperator{\MDM}{\overline{\cM}_g^{\mbox{{\fontsize{5}{12}\selectfont DM}}}}
\DeclareMathOperator{\MSat}{\overline{\cM}_g^{\mbox{{\fontsize{5}{12}\selectfont Sat}}}}
\DeclareMathOperator{\ASat}{\overline{\cA}_g^{\mbox{{\fontsize{5}{12}\selectfont Sat}}}}
\newcommand{\HSat}{H^{\mbox{{\fontsize{5}{12}\selectfont Sat}}}}
\DeclareMathOperator{\Mgo}{\cM_g^{o}}
\newcommand{\Sat}{\overline{\tau}}
\DeclareMathOperator{\Stab}{Stab}

\newcommand{\Ug}[1]{\mathcal{U}_{g,#1}}

\newcommand{\Gr}{\mathbb{G}}

\newcommand{\V}{\mathcal{V}}

\nc{\fA}{{\mathfrak{A}}}
\nc{\fB}{{\mathfrak{B}}}
\nc{\fC}{{\mathfrak{C}}}
\nc{\fD}{{\mathfrak{D}}}
\nc{\fE}{{\mathfrak{E}}}
\nc{\fF}{{\mathfrak{F}}}

\begin{document}

\title[Holomorphic $1$-forms on the moduli space of curves]{Holomorphic $1$-forms on the moduli space of curves}

\date{April 15, 2023}

\author{Filippo Francesco Favale}
\address{Dipartimento di Matematica,
	Universit\`a degli Studi di Pavia,
	Via Ferrata, 5
	I-27100 Pavia, Italy}
\email{filippo.favale@unipv.it}

\author{Gian Pietro Pirola}
\address{Dipartimento di Matematica,
	Universit\`a degli Studi di Pavia,
	Via Ferrata, 5
	I-27100 Pavia, Italy}
\email{gianpietro.pirola@unipv.it}

\author{Sara Torelli}
\address{Institut für Algebraische Geometrie,        
	Leibniz Universität Hannover, 
    Welfengarten 1,     
    30167 Hannover, Germany}
\email{torelli@math.uni-hannover.de}

\thanks{
\textit{2010 Mathematics Subject Classification}: Primary:  14H15; Secondary: 32L10\\
\textit{Keywords}: Moduli Space, $1$-forms, Extension of sections, Positivity }


\maketitle

\begin{abstract}
Since the sixties it is well known that there are no non-trivial closed holomorphic $1$-forms on the moduli space $\cM_g$ of smooth projective curves of genus $g>2$.
In this paper we strengthen such result proving that for $g\geq 5$ there are no non-trivial holomorphic $1$-forms. With this aim, we prove an extension result for sections of locally free sheaves $\cF$ on a projective variety $X$. More precisely, we give a characterization for the surjectivity of the restriction map
$\rho_D:H^0(\cF)\to H^0(\cF|_{D})$
for divisors $D$ in the linear system of a sufficiently large multiple of a big and semiample line bundle $\cL$. Then, we apply this to the line bundle $\cL$ given by the Hodge class on the Deligne Mumford compactification of $\cM_g$.
\end{abstract}


\section*{Introduction}
Let $X$ be a $n$-dimensional smooth irreducible projective variety defined over an algebraically closed field $\rm k$. We will say that a vector bundle $\cF$ over $X$ is \textit{liftable with respect to a line bundle $\cL$} (or \textit{$\cL$-liftable} in short) if there exists a positive integer $m_0$ such that  the restriction map  $$\rho_D:H^0(\cF)\to H^0(\cF|_{D})$$ is an isomorphism for any divisor $D\in |\cL^m|$ and for $m\geq m_0$ (see Definition \ref{def:liftable}). Surjectivity for $m$ large enough, is not garanteed, in general: further positivity assumptions on $\cL$ are needed. For instance,  $\cF$ is $\cL$-liftable as soon as $\cL$ is ample, by Serre’s criterions of vanishing and duality. One can furthermore relaxes this  up to $(n-2)$-ampleness (see \cite[Definition 1.3]{Som2}).
The first intent of this paper is to characterize  $\cL$-liftability for big and semiample line bundles.
We recall that $\cL$ is semiample  if for  some suitable $d>0$,  $\varphi_{|\cL^{d}|}:X\to \mathbb{P}H^0(\cL^{d})^*=\PP^N$  is a morphism. 
Furthermore, $\cL$ semiample is $(n-2)$-ample if has no divisors contracted to points by $\varphi_{|\cL^{d}|}$. 
We will show that, in the general case, the divisors contracted to points play a crucial role in Theorem \ref{thm:doubled} which can be stated as:
\begin{theorem*}{\textbf{A}}
~\\
Let $\cL$ be a big and semiample line bundle on $X$ and let $E$ be the divisor of $X$ contracted to points by $\varphi_{|\cL^d|}$ for $d$ large enough. Then a locally free sheaf $\cF$ on $X$ is $\cL$-liftable  if and only if for all $m>0$ the maps $H^0(\cF)\to H^0(\cF(mE))$ are surjective.
\end{theorem*}
The proof uses the Theorem of formal functions (\cite{HAG}*{III.11}). 
This theorem holds for an algebraically closed field, while the other result that we will present form Section \ref{SEC:EXAMPLES} onward, will essentially be over the complex numbers.
It is not surprising then that in the classical case,  that is when $\rm k$
is the field of complex numbers, the above statement translates into a sort of a concavity result. We borrow the terminology from complex analysis and geometry (see, for  example \cites{Andr,AG,Som1}) and say that $\cF$ is $\cL$-concave if for any divisor 
$D\in |\cL^a|$ with $a\geq 1$ and any open connected neighborhood $U$ of $D$ the restriction map $$\rho_U:H^0(\cF)\to H^0(\cF|_{U})$$ 
is surjective and therefore an isomorphism (see Definition \ref{def:concavity}). The open subset $U$, indeed, behaves in a similar way of a concave set in an analytic space (\cite{Andr}).
\\

We have the following (see Theorem \ref{THM:EXTENSIONFROMOPENS}).

\begin{theorem*}{\textbf{B}}
~\\
Let $X$ be a smooth complex projective variety and let $\cL$ be a big and semiample line bundle on $X$. Then, a vector bundle $\cF$ on $X$ is $\cL$-liftable
if and only if it is $\cL$-concave.
\end{theorem*}

In Section $2$ we give examples of surfaces to add value the above results. We investigate more precisely the cotangent bundle $\Omega^1_S$ of a smooth projective surface $S$ and show that such a sheaf can be either $\cL$-concave or not with respect to a big and semiample line bundle $\cL$.
We furthermore raise some questions about surfaces in the Noether-Lefschetz locus of $\PP^3$ (see Question \ref{QUE:SURFINP3}).
The importance of the cotangent bundle in this paper is much deeper and will appear evident in a moment. 

We are aware, also in view of the results of Totaro \cite{Tot} and Ottem \cite{Ott}, that it could be really  interesting to drop the assumption of semiampleness. This seems to us technically difficult at the moment and not necessary to tackle the problem that motivated all these studies.
\\ 

Let us introduce our motivating problem. Let $\pi:\cC\to B$ be a smooth holomorphic family of compact Riemann surfaces of genus $g$. During the preparation of \cite{BCFP}, Indranil Biswas explained to the second author of this article that $\cC^\infty$-families of projective structures on $C_t=\pi^{-1}(t),$ with $t\in B,$ are in one to one correspondence with $\overline \partial$-closed $\cC^\infty$ $(1,1)$-forms on $B$ with fixed cohomology class, modulo {\underline{holomorphic $(1,0)$-forms}} of $B$. For details, see Section $3$ of \cite{BCFP}. He raises then the problem of the existence of global holomorphic forms on $\cM_g$, the moduli space of compact Riemann surfaces, that is of smooth complex projective curves, of genus $g$. 
\\

It is well known, at least since Mumford \cite{Mu}, that there
are no closed holomorphic $1$-forms on $\cM_g$, and a proof of this will be outlined in Section $3$. We could not find any result in literature concerning non-closed holomorphic forms.																			
Our result, which can be seen as a concavity result, is the following (see Theorem \ref{THM:NOONEFORMS}).
\begin{theorem*}{\textbf{C}}
~\\
Let  $\Mgo\subset \cM_g$ be the smooth locus. Then for $g\geq 5$, $\Mgo$ has no holomorphic $1$-forms, that is $H^0(\Omega^1_{\Mgo})=0$.
\end{theorem*}

The proof of the theorem uses the Deligne-Mumford compactification $\MDM$ and the Satake map
$\Sat:\MDM\to \ASat$. Since  $\MSat:=\Sat(\MDM)$ is a projective variety, we intersect $\MSat$  with suitable general $3g-5$ and  $3g-4$ hyperplanes, respectively. By taking the inverse image on $\MDM,$ we reduce our problem  to curves and surfaces in $\MDM.$  We call these, respectively, $H$-surfaces and $H$-curves.
It is easy to show that for $g>3$ the general $H$-curve is contained in $\Mgo$   and  that a general  $H$-surface intersects the boundary of $\MDM$ only on $\Delta_1,$ the locus of stable curves with an elliptic tail. 
We apply our Theorem \ref{thm:doubled} to a general $H$-surface $S$ and $\cL=\cO_S(C)$, with $C$ general $H$-curve contained in $S$. 
Using the fact that the contracted divisor is exaclty $E=\Delta_1\cap S$, we show that $\Omega^1_{S}$ is  $\cL$-liftable and $H^0(\Omega^1_{S})=0.$
\\

We strengthen the above result by proving the following theorem which has the flavour of a concavity result (see Theorem \ref{MgCON}).
\begin{theorem*}{\textbf{D}}
~\\
Let $C$ be a $H$-curve and let $U\subset \Mgo$  be a connected open neighborhood of $C$
for the classical topology. Then, for $g\geq 5$, $H^0(\Omega^1_U)=0$.
\end{theorem*}

Our last result, contained in Subsection \ref{SUBS:MARKED}, is an extension of Theorem \ref{THM:NOONEFORMS} to the case of moduli of marked curves. More precisely, if $\cM_{g,n}^o$ is the smooth locus of $\cM_{g,n}$, we have the following (see Theorem \ref{THM:NOONEFORMSMARKED}).
\begin{theorem*}{\textbf{E}}
~\\
Let $g\geq 5$. Then $H^0(\Omega^1_{\cM_{g,n}^o})=0$, for all $n\geq 0$.
\end{theorem*}


The proof of Theorem E is straightforward,
but uses the extra-ingredient of the infinitesimal variation of Hodge structures. We would also like to mention that the existence of possibly non closed holomorphic forms in a neighborhood of a compact curve plays a subtle role in the infinitesimal variation of its periods (\cites{PT,GST,GT,FT}).  Similar results should hold at least for many families of curves, for instance the case of smooth plane curves is treated implicitly in \cite{FNP} and \cite{PT}. 
\smallskip

As a remarkable consequence, Theorems C and E solve the correspondent problems at the level of moduli stack of curves (possibly, with marked points), by interpreting $\cM_{g,n}^o$ as an open subset of the stack.

\begin{corollary*}
For $g\geq 5$, the moduli stack of curves with $n\geq 0$ marked points has no holomorphic $1$-forms.
\end{corollary*}

To conclude, the most interesting problem arising from our result would probably be to consider holomorphic $p$-forms on the moduli space of curves. The methods used for $1$-forms seem to us insufficient at the moment to deal with these more general cases.

\subsection*{Acknowledgements}

\begin{small}
We would like to thank Indranil Biswas for introducing this interesting problem to us, Francesco Bonsante and Riccardo Salvati Manni for helpful remarks and discussions on the Teichmuller space and on the compactification of the moduli space of curves and Alessandro Ghigi for enlightening explainations on the relations between concavity and extensions of holomorphic functions. The authors want to express their gratitude to the anonymous referee for the helpful remarks and suggestions. \\
All authors are partially supported by INdAM - GNSAGA. The first named author was partially supported by ``2017-ATE-0253'' (Dipartimento di Matematica e Applicazioni - Universit\`a degli Studi di Milano-Bicocca) and is now partially supported by INdAM-GNSAGA Project "Classification Problems in Algebraic Geometry: Lefschetz Properties and Moduli Spaces” (CUP$\_$E55F22000270001).
The second named author is partially supported by PRIN 2017\emph{``Moduli spaces and Lie theory''} and by (MIUR): Dipartimenti di Eccellenza Program (2018-2022) - Dept. of Math. Univ. of Pavia. The third author is supported by Alexander von Humboldt Foundation.
\end{small}

\section{Surjectivity of restriction maps}
\label{SEC:MAINRES}

Through all this section $X$ will be a smooth projective variety of dimension $n\geq 2$ over an algebraically closed field $\rm k$.  For any $\cL$ big and semiample line bundle on $X$, let  $d_0\in \mathbb{N}$ be a positive integer such that for any $d\geq d_0$ the morphism 
$$\psi_{|\cL^{d}|}:X\to \mathbb{P}H^0(\cL^{d})^*=\PP^{N_d}$$ 
is birational onto its image and the map $\psi_{|\cL^{d}|}:X\to \psi_{|\cL^{d}|}(X)$ does not depend on $d$. We set $Y=\psi_{|\cL^{d}|}(X)$, $\psi: X\to Y$ the induced  morphism, $E\subset X$ the divisor contracted to points and $E_i$ the connected divisor contracted to the point $p_i$. Notice that in particular the divisor $E$ does not depend on $d$.

\begin{remark}
\label{REM:SEMIAMPLE}
Notice that $\cL=\psi^*(\cL')$ where $\cL'$ is an ample line bundle on $Y$. Indeed, by assumption the map induced by $\cL^d$ and $\cL^{d+1}$ are the same and so $\cL^{d}=\psi^*\cO_{\bP^{N_d}}(1)_{|Y}$, for any $d\geq d_0$. Therefore,  $\cL=\cL^{d+1}\otimes \cL^{-d}=\psi^*( \cO_{\bP^{N_{d+1}}}(1)_{|Y}\otimes\cO_{\bP^{N_{d}}}(1)^{-1}_{|Y})$ as claimed.
\end{remark}


\smallskip
Let $\cF$ be a locally free sheaf on $X$. For $a$ large enough consider $D_a\in |\cL^a|$ and take the short exact sequence induced by $\cO_{X}(-D_a)\subset \cO_X$ twisted with $\cF$
\begin{equation}\label{sxs:Extendable}
0\to \cF(-D_a)\to \cF\to \cF_{|D_a}\to 0.
\end{equation}

\begin{definition}
\label{def:liftable} We say that $\cF$ is $\cL$-liftable if the map 
	\begin{equation}\label{map:restriction} \rho_a\,:\,H^0(\cF)\to H^0(\cF_{|D_a})\end{equation}
	induced by \eqref{sxs:Extendable} is an isomorphism, for all $a$ large enough and any $D_a\in |\cL^a|$. 
\end{definition}
Consider also the short exact sequence 
\begin{equation}\label{sxs:contractedDivisor}
0\to \cF\to \cF(mE)\to \cF(mE)_{|_{mE}}\to 0.
\end{equation}

This section is dedicated to prove the following Theorem. 
\begin{theorem}\label{thm:doubled}
	$\cF$ is $\cL$-liftable if and only if  for $m\geq 0$, $\tau_m\,:\,H^0(\cF)\to H^0(\cF(mE))$ induced by \eqref{sxs:contractedDivisor} is surjective.
\end{theorem}
Observe that injectivity holds as soon as $a$ is large enough since $H^0(\cF(-D_a))=0$. Hence, $\cL$-liftability is a property concerning the surjectivity of that map. We thus have to study the injectivity of the map $H^1(\cF(-D_a))\to H^1(\cF)$, which is equivalent, by Serre duality, to the surjectivity of $H^{n-1}(\cE)\to H^{n-1}(\cE\otimes \cL^a)$, where we denote $\cE=\cF^*\otimes\omega_X$. We first compute $H^{n-1}(\cE\otimes \cL^a)$. 

\begin{lemma}
	\label{LEM:LERAY} For $a$ large enough $H^{n-1}(\cE\otimes \cL^a)\simeq H^0(R^{n-1}\psi_*\cE).$
	Moreover, $\cG=R^{n-1}\psi_*\cE$ is a sum of skyscraper sheaves $\cG_i$ supported on the images $p_i$ of the divisors $E_i$ contracted to points by $\psi$.
\end{lemma}

\begin{proof} 
Recall that $\cL=\psi^*(\cL')$ with $\cL'$ ample as observed in Remark \ref{REM:SEMIAMPLE}. Applying the projection formula, we have
$$R^p\psi_*(\cE\otimes \cL^a)=R^p\psi_*(\cE\otimes \psi^*(\cL')^a)=R^p\psi_*(\cE)\otimes (\cL')^a.$$
As $a$ is large enough we can apply Serre criterion to obtain 
$$H^i(R^p\psi_*(\cE\otimes \cL^a))=H^i(R^p\psi_*(\cE)\otimes (\cL')^a)=0,$$
for $i>0$ and any $p$. Then, all terms in the Leray spectral sequence are zero except for $H^0(R^{n-1}\psi_*(\cE\otimes \cL^a))$ that is mapped to zero by the differential. So we get the isomorphism
$$H^{n-1}(\cE\otimes\cL^a)\simeq H^0(R^{n-1}\psi_*(\cE\otimes \cL^a)).$$
Consider the sheaf $\cG=R^{n-1}\psi_*(\cE\otimes \cL^a)$. If $y\in Y$ is different from $p_i$ for all $i$, then the fiber of $\psi$ over $y$ is a subvariety of $X$ of codimension at least $2$. Hence, the stalk of $\cG$ in $y$ is $0$. This proves also that $\cG$ is a sum of skyscraper sheaves $\cG_i$ with support on the points $p_i$. But now $\cL|_{E_i}=(\psi^*\cL')|_{E_i}=\OO_{E_i}$, so $\cG=R^{n-1}\psi_*(\cE\otimes \cL^a)=R^{n-1}\psi_*\cE$.
\end{proof}

Observe now that for $a\geq d_0$, for all $D_a\in |\cL^{a}|$ we have $\cF(mE)_{|D_a}\simeq \cF_{|D_a}$. Indeed, every $D_a$ is the inverse image of an hyperplane and so the general $D_a$ is disjoint from $E$. Therefore, $\cO_{D_a}(mE)\simeq \cO_{D_a}$. Notice that we can conclude the same for all $D_a$ by Seesaw Theorem.
Thus, using the map $\OO_{X}(-mE)\hookrightarrow \OO_X$, the short exact sequence \eqref{sxs:Extendable} and its twist by $\cO_{X}(mE)$ for $m\geq 0$ fit into the commutative diagram
\begin{equation}
\label{DIAG:SHEAVES}
\xymatrix{
	0 \ar[r] & 
	\mathcal{F}\otimes \cL^{-a} \ar[r]\ar@{^{(}->}[d] &
	\cF \ar[r]^-{\rho_a}\ar@{^{(}->}[d]&  \cF|_{D_a}\ar[d]^-{\simeq} \ar[r]
	& 0 \\
	0 \ar[r] &
	\mathcal{F}(mE)\otimes \cL^{-a} \ar[r]&
	\cF(mE) \ar[r] &
	\cF|_{D_a} \ar[r]&
	0 
}
\end{equation}
that induces in cohomology 
\begin{small}
\begin{equation}
\label{DIAG:DEFRHOTAU}
\xymatrix@C=15pt{
	H^0(\cF\otimes \cL^{-a}) \ar@{^{(}->}[r]\ar@{^{(}->}[d] &
	H^0(\cF) \ar[r]^-{\rho_a} \ar@{^{(}->}[d]_-{\tau_m} & 
	H^0(\cF|_{D_a}) \ar[r] \ar[d]^{\simeq} & 
	H^1(\cF\otimes \cL^{-a}) \ar[r]^-{\eta_a} & 
	H^1(\cF) \\
	H^0(\cF(mE)\otimes \cL^{-a}) \ar@{^{(}->}[r] & 
	H^0(\cF(mE)) \ar[r]^-{\rho_a'} & 
	H^0(\cF|_{D_a}). \\
}
\end{equation}
\end{small}
Let $\tau_m^\ast$ and $\eta_a^*$ be the maps obtained from $\tau_m$ and $\eta_a$ by Serre duality, respectively. Then, $\tau_m^\ast$ fits into the following exact sequence induced by $\OO_{X}(-mE)\hookrightarrow \OO_X$ twisted by $\cE$
\begin{equation}\label{es:delta}\cdots \to H^{n-1}(\cE)\xrightarrow{\delta_m} H^{n-1}(\cE|_{mE})\to H^n(\cE(-mE))\xrightarrow{\tau_m^*} H^n(\cE)\to \cdots.\end{equation}

We summarize two results which stem from the above discussion in the following Lemma.

\begin{lemma}
	\label{LEM:SurjectiveReduction}
The following equivalent conditions hold:
	\begin{enumerate}
		\item the surjectivity of $\rho_a$ is equivalent to the surjectivity of $\eta_a^*:H^{n-1}(\cE)\to H^{n-1}(\cE\otimes \cL^{a});$\smallskip
		\item the surjectivity of $\tau_m$ is equivalent to the surjectivity of $\delta_m:H^{n-1}(\cE)\to H^{n-1}(\cE|_{mE}).$
	\end{enumerate}
\end{lemma}

\begin{proof}[Proof of Theorem \ref{thm:doubled}]   Assume first that $\rho_a$ is surjective, for $a>>0$. We prove that $\tau_m$ is surjective, for any $m>0$. By Diagram \eqref{DIAG:DEFRHOTAU}, 
$$
\xymatrix{
	& H^0(\cF) \ar[d]_-{\tau_m} \ar[r]^-{\rho_a}& H^0(\cF|_{D_a}) \ar[d]^-{\simeq } \\
	H^0(\cF(mE)\otimes \cL^{-a})\ar@{^{(}->}[r] &
	H^0(\cF(mE))\ar[r]^-{\rho_a'} & 
	H^0(\cF|_{D_a}).
}
$$
As $\rho_a$ is surjective,  also $\rho_a'\circ\tau_m$ is surjective. But now the kernel of $\rho_a'$ is $H^0(\cF(mE)\otimes \cL^{-a})$ and it  is trivial for $a$ large enough. Hence $\tau_m$ is surjective. 
\smallskip
	
Let us now prove the other implication. Assume that $\tau_m$ is surjective, for any $m>0$. We prove that $\rho_a$ is surjective, for $a>>0$. Equivalently, by Lemma \ref{LEM:SurjectiveReduction}, it is enough to prove that $\eta_a^*:H^{n-1}(\cE)\to H^{n-1}(\cE\otimes \cL^{a})$ is surjective. By Lemma \ref{LEM:LERAY} we can write $H^{n-1}(\cE\otimes \cL^{a})=H^0(R^{n-1}\psi_*\cE)=\bigoplus_i H^0((R^{n-1}\psi_*\cE)_{p_i})$ because $(R^{n-1}\psi_*\cE)_{p_i}$ is skyscraper supported on the points $p_i$ images of contracted divisors and so the map $\eta_a^*$ from Lemma \ref{LEM:SurjectiveReduction} up to isomorphism is  
\begin{equation}
\label{map:formulalimit}
\eta_a^*:H^{n-1}(\cE)\to \bigoplus_i H^0((R^{n-1}\psi_*\cE)_{p_i}).
\end{equation}
To prove the surjectivity of \eqref{map:formulalimit} we use the machinary of inverse limits and the Theorem of formal functions (see \cite{GRO}[(4.1.5),(4.2.1)] and \cite{HAG}[III, Theorem 11.1]). This Theorem gives the isomorphisms
$$(R^{n-1}\psi_*\cE)_{p_i}\hspace{-3mm}\sphat \,\simeq \varprojlim H^{n-1}(\cE|_{kE_i})$$
and defines then a map 
\begin{equation}
\label{EQ:WIDEETA}
\widehat{\eta}_a^*:H^{n-1}(\cE)\to \bigoplus_i \varprojlim H^{n-1}(\cE|_{kE_i}).
\end{equation}
We will prove that $\widehat{\eta}_a^*$ is surjective and that this implies that $\eta_a^*$ is surjective.
\smallskip 

To be more explicit consider the following two inverse systems of $\rm k$-vector spaces $\fB$ and $\fC$. The first one, $\fB=(B_k)_{k\in \mathbb{N}}$, is the constant inverse system with $B_k=H^{n-1}(\cE)$ and maps $b_k=\id: B_k\to B_{k-1}$, for all $k$. In order to define the second one, $\fC=(C_k)_{k\in \mathbb{N}}$, we set  $C_k=H^{n-1}(\cE|_{kE})$ and we construct the maps $c_k:C_k\to C_{k-1}$ as follows. Consider for any $k\geq 1$ the commutative diagram
\begin{equation}
\label{DIAG:INVSYST}
\xymatrix{
	0 \ar[r]& 
	\cE(-kE) \ar[r]\ar@{^{(}->}[d] &
	\cE \ar[r] \ar[d]_{\id}&
	\cE|_{kE} \ar[r]\ar[d]^-{c'_k} &
	0\\
	0 \ar[r]& 
	\cE(-(k-1)E) \ar[r] \ar@{->>}[d]&
	\cE \ar[r] &
	\cE|_{(k-1)E} \ar[r] &
	0\\
	&
	\cE(-(k-1)E)|_{E}
}
\end{equation}
where the two rows are defined by $\OO_X(-lE)\subset \OO_X$ twisted with $\cE$ for $l=k$ and $l=k-1$ respectively, the first column is given by $\cO_X(-E)\subset \cO_X$, the second column by the identity and the map $c'_k$ in the third column is that one letting the diagram commutative. As the vertical arrow in the first column is an injection and that one in the second column is the identity, 	$c'_k$ is surjective and its kernel is isomorphic to $\cE(-(k-1)E)|_{E}$. Hence, the last column induces the exact sequence
\begin{equation}
\label{EQ:DEFPSIK}
\cdots\to H^{n-1}(\cE(-(k-1)E)|_{E})\to H^{n-1}(\cE|_{kE})\xrightarrow{c_k} H^{n-1}(\cE|_{(k-1)E})\to 0    
\end{equation}
where the last $0$ follows as $\cE(-(k-1)E)|_{E}$ is supported on a divisor. So the map $c_k$ are epimorphisms.

Now we want to prove that there is a surjection between the limits of the inverse systems $\fB$ and $\fC$ induced by the morphism $\delta=(\delta_m)_{m\in \bN}:\fB\to \fC$ where
\begin{equation}\delta_m:H^{n-1}(\cE)\to H^{n-1}(\cE|_{mE})\end{equation} is the morphism in Lemma \ref{LEM:SurjectiveReduction}. These are surjective for $m\geq 0$ by Lemma \ref{LEM:SurjectiveReduction}, since by assumption 
$\tau_m:H^0(\cF)\to H^0(\cF(mE))$ is surjective for $m\geq 0$. In particular, we have a surjective map between the limits
\begin{equation}
\label{EQ:SURJ}
\xymatrix{
	H^{n-1}(\cE)=\varprojlim\fB \ar@{->>}[r] & \varprojlim\fC = \varprojlim H^{n-1}(\cE|_{kE}).
}
\end{equation}



Observe that $\varprojlim H^{n-1}(\cE|_{kE})=\bigoplus_{i}\varprojlim H^{n-1}(\cE|_{kE_i})$. By \cite{HAG}[III, Proposition 8.5 and Theorem 11.1], the vector space $H^{n-1}(\cE|_{kE_i})$ has a natural structure of $\OO_{Y,p_i}$-module, which is compatible with the structure of $\rm k$-vector space. 
Thus we can apply the Theorem of formal functions (see \cite{GRO}[(4.1.5),(4.2.1)] and \cite{HAG}[III, Theorem 11.1]) to $R^{n-1}\psi_*\cE$ to conclude that 
\begin{equation}
\label{EQ:FORMALFUNC}
(R^{n-1}\psi_*\cE)_{p_i}{\hspace{-3mm}\sphat}\,\,=\,\varprojlim H^{n-1}(\cE|_{kE_i})=\varprojlim\fC.
\end{equation}

Using Equations \eqref{EQ:SURJ} and \eqref{EQ:FORMALFUNC}, one can define the morphism 
$$\widehat{\eta}_a^*:H^{n-1}(\cE)\to \bigoplus_i (R^{n-1}\psi_*\cE)_{p_i}{\hspace{-3mm}\sphat}$$
in \eqref{EQ:WIDEETA}, which is surjective by construction.



In order to conclude that
$$\eta_a^*:H^{n-1}(\cE) \to \bigoplus_i H^0((R^{n-1}\psi_*\cE)_{p_i})$$
is surjective, notice that $(R^{n-1}\psi_*\cE)_{p_i}{\hspace{-3mm}\sphat}\,\,$ is naturally isomorphic to $H^0((R^{n-1}\psi_*\cE)_{p_i})$. Indeed, one has
$$(R^{n-1}\psi_*\cE)_{p_i}{\hspace{-3mm}\sphat} \, =\, (R^{n-1}\psi_*\cE)_{p_i}\otimes \widehat{\OO_{p_i}}= (R^{n-1}\psi_*\cE)_{p_i}\otimes \OO_{p_i}= (R^{n-1}\psi_*\cE)_{p_i}= H^0((R^{n-1}\psi_*\cE)_{p_i})$$
where the second equality follows from the isomorphism $\widehat{\OO_{p_i}}\simeq \OO_{p_i}$ (which holds since $\OO_{p_i}$ is Artinian) and the last equality follows since the sheaf is supported on the point.
\end{proof}

\subsection{A concavity result}

~\\ 
From now on we assume $\rm k=\bC$. We will prove a result, which is of analytic kind, by applying the algebraic results stated in Theorem \ref{thm:doubled}. Let $\cL$ be a line bundle on a smooth projective variety $X$ of dimension $n$. Motivated by the classical notion of concavity (see \cite{Andr,AG,Som1}), we give the following definition.

\begin{definition}\label{def:concavity}
We say that a locally free sheaf $\cF$ on $X$ is {\bf $\cL$-concave} if for any $D\in |\cL^a|$ with $a\geq 1$ and any open connected neighborhood  $U$ of $D$, with respect to the analytical topology, the restriction map gives the equality $H^0(\cF)=H^0(\cF|_U)$.
\end{definition}

\begin{theorem}
\label{THM:EXTENSIONFROMOPENS} Let $\cL$ be a big and semiample line bundle. Then, a locally free sheaf $\cF$ is $\cL$-liftable if and only if $\cF$ is $\cL$-concave.
\end{theorem}
\begin{proof}
Assume first that $\cF$ is $\cL$-liftable. Take $D\in |\cL^a|$ and let $U$ be an open connected neighborhood of $D$ with respect to the analytical topology.
Since $\cF$ is $\cL$-liftable, there is $m_0\in \mathbb{N}$ such that the restriction map induces the isomorphism $H^0(\cF)\simeq H^0(\cF|_{mD})$ for all $m\geq m_0$. In particular, the restriction map induces isomorphisms $H^0(\cF|_{m_0D})\simeq H^0(\cF|_{(m_0+k)D})$ for $k\geq 0$.  
We have to show that the restriction $H^0(\cF)\to H^0(\cF|_U)$ is surjective. Fix a section $\omega\in H^0(\cF|_U)$ and let $\omega_m\in H^0(\cF|_{mD})$ be its restriction.
If $\alpha\in H^0(\cF)$ is the lift of $ \omega_{m_0}$, we have that, for all $m\geq m_0$, $\alpha|_U-\omega\in H^0(\cF|_U)$ restricts to zero in $H^0(\cF|_{mD})$. 
Thus, the series expansion of $\alpha|_U-\omega$ in a local coordinate neighborhood of the generic point of $U$ vanishes. Therefore since $U$ is connected $\alpha|_U-\omega=0$.

Now assume that $\cF$ is not $\cL$-liftable. Since $\cL$ is big and semiample, by Theorem \ref{thm:doubled} there is a non zero effective divisor $E$ contracted by $\varphi_{|\cL^d|}$ for $d>>0$, and a section $\alpha \in H^0(\cF(mE))\setminus H^0(\cF)$ for some integer $m>0$.  Set $U=X\setminus E$ and  $D\in |\cL^d|$ such that $D\cap E=\emptyset$. Then $U$ is an open connected neighborhood of $D$. By construction the restriction $\alpha|_U$ defines a section of $H^0(\cF|_U)$ that cannot be extended to $H^0(\cF)$.
\end{proof}



\section{Examples in dimension two}
\label{SEC:EXAMPLES}

In this section we analyse Theorem \ref{thm:doubled} when the variety is a surface $S$ over an algebraically closed field of characteristic $0$ and $\cF=\Omega_S^1$ is its cotangent bundle. Theorem \ref{thm:doubled}, in this framework, can be restated here as follows:
\begin{center}
\textit{``Let $S$ be as surface and let $\cL$, and $E$ be as in Theorem \ref{thm:doubled}. \\
Then $\Omega_S^1$ is not $\cL$-liftable if and only if there exists $m>0$ such that  $h^0(\Omega^1_S{(mE)})>h^0(\Omega^1_S)$."}
\end{center}
We will  give examples of surfaces $S$ for which $\Omega_S^1$ is $\cL$-liftable and cases for which it is not.

\subsection{The non liftable case: projective bundles over curves and coverings}
\label{SUBSEC:PROJBUNDLES}
~\\ 
We will use well known results about projective bundles. The reader can refer to \cite[Chapter V.2]{HAG} for details. Let $B$ be a smooth projective curve of genus $g\geq 2$. We fix a globally generated line bundle $M$ of degree $d>0$ with $h^1(M)>0$. Then one has $0<d\leq 2g-2$ and $d=2g-2$ if and only if $M=\omega_B$.
We can consider the vector bundle $\V=\OO_B\oplus M^{-1}$ on $B$. It is the only decomposable normalized\footnote{Recall that a vector bundle $\mathcal{E}$ on a curve $B$ is normalized if $h^0(\mathcal{E})\neq 0$ and, for all line bundle $M$ of negative degree, $h^0(\mathcal{E}\otimes M)=0$.} vector bundle of rank $2$ on $B$ such that $c_1(\V)=M^{-1}$. Let $S=\PP(\V)\xrightarrow{f} B$ and consider the section $\sigma:B\to S$ induced by $\V\twoheadrightarrow M^{-1}$. Its image is an effective curve $E$ which is isomorphic to $B$ via $\sigma$. Moreover, by construction, if $\cN$ is a line bundles on $B$, then $f^*(\cN)|_E$ corresponds to $\cN$ via $\sigma$. In particular, we have
$$\cO_E(E)=f^*(M^{-1})|_E$$ and so
$E^2=-d$. Set $\cL=\OO_S(E)\otimes f^*M$ and take $H\in|\cL|$. Notice that, by construction, $\cO_E(H)=\cO_E$ so $H\cdot E=0$.

\begin{proposition}
The line bundle $\cL$ is big and $|\cL|$ is base point free. Moreover, $E$ is contracted by $\varphi_{|\cL|}$ and $\Omega_S^1$ is not $\cL$-liftable.
\end{proposition}

\begin{proof}
Let $p\in S$. Since $|M|$ is base point free by assumption and $E+|f^*M|$ is a subsystem of $|\cL|$, we have that if $p$ is a base point of $|\cL|$, then necessarily $p\in E$. On the other hand, by projection formula
we have 
$$f_*(\cO_S(E))=\cV\qquad f_*(\cO_S(H-E))=M\qquad f_*(\cO_S(H))=\cV\otimes M$$
so 
$$H^0(\cO_S(H-E))=H^0(f^*M)\simeq H^0(M)\qquad H^0(\cO_S(H))\simeq H^0(\cV\otimes M)\simeq H^0(\cO_B)\oplus H^0(M).$$
Hence, from the exact sequence
\begin{equation}
\label{EXSEQ:HYPER}
0\to H^0(\OO_S(H-E))\to H^0(\OO_S(H))\xrightarrow{\alpha} H^0(\OO_{E}(H))\to H^1(\cO_S(H-E))\to \cdots
\end{equation}
one has that $\alpha$ is surjective since we have shown that $H^0(\cO_S(H-E))$ has codimension $1$ in $H^0(\cO_S(H))$ and $H^0(\cO_E(H))=H^0(\cO_E)$ is $1$-dimensional. This also shows that there exists a section $s$ of $\cL=\OO_S(H)$ that is not zero at any point of $E$. Hence $|\cL|$ is base point free as claimed. This also shows that $H$ is nef and, since $H^2=d>0$, we have that $H$ is big.
As observed before, we have $H\cdot E=0$, so the morphism $\varphi_{|\cL|}$ contracts $E$.


We have $R^1f_*\OO_S=0$ since the fibers of $f$ are projective lines. Since $S$ is ruled we have that $h^0(\Omega_S^1)=g$. Hence, in order to show that $\Omega_S^1$ is not $\cL$-liftable it is enough to show that $h^0(\Omega_S^1(E))>g$. Consider the relative cotangent sheaf $\Omega_{S/B}^{1}$. It is a line bundle on $S$ such that
$$\Omega^1_{S/B}=\omega_S\otimes f^*\omega_B^{-1}=\OO_S(-2E)\otimes f^*M^{-1}=\OO_S(-H-E).$$
If we twist the cotangent sequence by $E$ we obtain
$$0\to H^0(\OO_S(E)\otimes f^*\omega_B)\to H^0(\Omega_S^1(E))\to H^0(\Omega_{S/B}^1(E))=H^0(\OO_S(-H))=0$$
and so $H^0(\Omega_S^1(E))\simeq H^0(\OO_S(E)\otimes f^*\omega_B)$. Then, as $f_*\OO_S(E)=\cV$, we can write
$$h^0(\Omega_S^1(E))=h^0(\omega_B)+ h^0(\omega_B\otimes M^{-1})=g+h^1(M).$$ 
As $h^1(M)>0$ by assumption, we proved the claim.
\end{proof}

Now, we want to produce other examples for which the liftability property of the cotangent sheaf does not hold. Let $S$ as before and consider a generically finite projective morphism $\pi:\hat{S}\to S$ such that the branch divisor $D$ is smooth and different from $E$. Set $\hat{E}=\pi^*E$, $\hat{H}=\pi^*H$ and $\hat{\cL}=\pi^*\cL$.

\begin{proposition}
\label{PROP:FINITECOVER}
The line bundle $\hat{\cL}$ is big, $|\hat{\cL}|$ is base point free and $\varphi_{|\hat{\cL}|}$ contracts $\hat{E}$. Moreover, $\Omega_{\hat{S}}^1$ is not $\hat{\cL}$-liftable.
\end{proposition}

\begin{proof}
Since $\pi$ is surjective and generically finite, we have that $\hat{\cL}$ is big and that the linear system $|\hat{\cL}|$ is base point free. Since $\pi^*$ commutes with the intersection product we have $0=\pi^*(H\cdot E)=\pi^*(H)\cdot \pi^*(E)=\hat{H}\cdot \hat{E}$ so $\hat{E}$ is contracted by $\varphi_{|\hat{\cL}|}$. As the branch divisor of $\pi$ is different from $E$ we have that if $\eta\in H^0(\Omega_S^1(E))\setminus H^0(\Omega_S^1)$ then $\pi^*\eta\in H^0(\Omega_{\hat{S}}^1(E))\setminus H^0(\Omega_{\hat{S}}^1)$. This proves that $\Omega_{\hat{S}}^1$ is not $\hat{\cL}$-liftable.
\end{proof}

We conclude this subsection by constructing examples of cyclic covering of any $S$ as before. These give elliptic fibrations and surfaces of general type. We will use some results about cyclic coverings which can be found in \cite[Chapter I.17]{BPHV}.

Recall that $\cL$ is base point free. Then, for each $n\geq 1$, there exists a smooth irreducible curve $C_n\in |\cL^n|$ by Bertini's theorem (see \cite[Theorem 3.3.1]{Laz}) and using that $\varphi_{|\cL|}(S)$ has dimension $2$. Hence, for all $n\geq 1$, we can construct a cyclic covering $\pi:\hat{S}\to S$ of degree $n$ with branch $C_n$. Since $C_n$ is smooth and it does not intersect $E$ for all $n$, we can apply Proposition \ref{PROP:FINITECOVER} in order to prove that $\Omega_{\hat{S}}^1$ is not $\hat{\cL}$-liftable.

\begin{proposition}
Consider the cyclic coverings above with branch $C_n$. Then
\begin{enumerate}[(a)]
\item for all $n\geq 1$, $\hat{f}=f\circ \pi$ is a fibration;
\item the general fiber of $\hat{f}$ is smooth of genus $\frac{(n-1)(n-2)}{2}$;
\item $\hat{S}$ is an elliptic fibration for $n=3$ and is canonically polarized for $n\geq 4$.
\end{enumerate}
\end{proposition}

\begin{proof}
(a) The line bundle $\cL$ restricted to the fibers of $f$ has positive degree (more precisely, $H\cdot F=1$) and we have
$$\hat{f}_*\OO_{\hat{S}}=f_*(\pi_*\OO_{\hat{S}})= f_*\left(\bigoplus_{k=0}^{n-1}\cL^{-k}\right)= \bigoplus_{k=0}^{n-1}f_*(\cL^{-k})=f_*\OO_S=\OO_B.$$
Hence $\hat{f}$ is proper and surjective and has connected fibers as well, i.e. it is a fibration. 

(b) The general fiber $F$ of $S$ intersects $C_n$ transversally in $nH\cdot F=nE\cdot F=n$ points so $\hat{F}=\pi^{-1}(F)$ is a covering of $F$ totally ramified on $n$ points and unramified outside these $n$ points. From Riemann-Hurwitz we obtain that the genus of the general fiber $\hat{F}$ of $\hat{f}$ is
$$g(\hat{F})=\frac{(n-1)(n-2)}{2}.$$
Notice, in particular, that $g(\hat{F})=1$ if $n=3$.

(c) By $(b)$, $\hat{S}$ is an elliptic fibration for $n=3$. Assume now $n\geq 4$. We want to show that $\omega_{\hat{S}}$ is ample. As $\hat{S}$ is a cyclic covering of $S$ of order $n$ we have $\omega_{\hat{S}}=\pi^*(\omega_S\otimes\cL^{n-1})$. Since
$$\omega_S\otimes \cL^{n-1}=\OO_S((n-3)E)\otimes f^*(\omega_B\otimes M^{n-2}),$$
by \cite[Chapter V.2, Proposition 2.20]{HAG} we have that $\omega_S\otimes \cL^{n-1}$ is ample as soon as $n\geq 4$. Then, being $\pi$ finite and surjective, we have that $\omega_{\hat{S}}$ is ample and $\hat{S}$ is canonically polarized.
\end{proof}


\subsection{The liftable case: surfaces in the Noether-Lefschetz locus}
\label{SUBSEC:NOETHERLEFSCHETZ}
~\\ 
For this example we restrict to the case $k=\bC$. We analyze some surfaces $S$ in $\PP^3$ with $\Omega_S^1$ that is $\cL$-liftable for a suitable big and semiample line bundle $\cL$. In order to find interesting examples one needs to consider surfaces in the Noether-Lefschetz locus as, otherwise, all big line bundles on $S$ would be multiples of the hyperplane class (see \cite[Chapter I, Section 3.3]{Voi2} and \cite{Lop} for details).

More precisely, $S$ will be a very general surface in the Noether-Lefschetz locus of sextic surfaces which contains a general quartic plane curve $E$. These surfaces have Picard rank $2$ and $\NS(S)\simeq\Pic(S)$ is spanned by the hyperplane class $H$ and by $E$ (see \cite{Lop}). An extremal ray for the cone of effective curves on $S$ is $E$ itself as $E^2=-4$ whereas the other one is the residual intersection  $R$ in $S$ of the hyperplane containing $E$. By construction, $R$ is a conic with self intersection $R^2=-6$. It is easy to see that $\OO_S(H+E)$ and $\OO_S(4H-E)$ (contracting $E$ and $R$ respectively) are, up to multiples, the only line bundles which are big, semiample (actually, globally generated) but not ample. Moreover, the classes of $H+E$ and $4H-E$ span the nef cone of $S$.

\begin{proposition}
The cotangent bundle $\Omega_S^1$ is $\cL$-liftable for $\cL$ equal to $\OO_S(H+E)$ or $\OO_S(4H-E)$.
\end{proposition}

\begin{proof}
We will prove $\cL$-liftability of $\Omega_S^1$ for $\cL=\OO_S(H+E)$. With the same techniques, one can prove the result for $\OO_S(4H-E)$. 
\smallskip

Since $E$ is a general plane curve and $S$ is assumed to be very general, the morphism $\varphi_{|\cL|}$ cannot contract other curves besides $E$ (since the effective cone is spanned by $E$ and $R$).
As $h^0(\Omega_S^1)=0$, by Theorem \ref{thm:doubled}, we have to show that $h^0(\Omega_S^1(mE))=0$ for all $m\geq 0$. Notice that the sequence
$\{h^0(\Omega_S^1(mE))\}_{m\geq 0}$ is non decreasing. We claim that it is stationary from $m=1$ onwards. To see this we consider the cotangent bundle sequence of $E$ in $S$ and twist it with $mE$, i.e. the sequence $0\to \OO_E((m-1)E)\to \Omega_S^1(mE)|_E \to \omega_E(mE)\to 0$. This yields
$$0\to H^0(\OO_E((m-1)E))\to H^0(\Omega_S^1(mE)|_E)\to H^0(\omega_E(mE))\to \cdots.$$
For $m\geq 2$ we have that both $\deg_E((m-1)E)$ and $\deg_E(\omega_E(mE))$ are negative so we have $h^0(\Omega_S^1(mE)|_E)=0$. Hence, form the exact sequence
$$0\to \OO_S((m-1)E)\to \OO_S(mE)\to \OO_E(mE)\to 0$$
twisted by $\Omega_S^1$ we get $h^0(\Omega_S^1(E))=h^0(\Omega_S^1(mE))$
for all $m\geq 1$. Hence, it is enough to show that $h^0(\Omega_S^1(E))=0$.

The restriction of the Euler sequence on $\PP^3$ to $S$ twisted by $E$ yields an exact sequence
$$0\to H^0(\Omega_{\PP^3}^1(E)|_S)\to H^0(\OO_S(-H+E))^{\oplus 4}\to H^0(\OO_S(E))\to \cdots$$
which gives us $h^0(\Omega_{\PP^3}^1(E)|_S)=0$ as $-H+E=-R$ is not effective. If we denote by $N_{S/\PP^3}^*$ the conormal bundle of $S$ in $\PP^3$, we can write the cotangent sequence of $S$ in $\PP^3$ twisted by $\OO_S(E)$ as
$$0\to N_{S/\PP^3}^*(E)\to \Omega_{\PP^3}^1(E)|_S\to \Omega_S^1(E)\to 0.$$
As $h^0(\Omega_{\PP^3}^1(E)|_S)=0$, this gives an injection $H^0(\Omega_S^1(E))\hookrightarrow H^1(N^*_{S/{\PP^3}}(E))$. Hence, we are done if we prove that $H^1(N^*_{S/{\PP^3}(E)})=0$. Notice that
$$H^1(N^*_{S/{\PP^3}}(E))=H^1(\OO_S(-6H+E)).$$
The divisor $6H-E$ is ample since it is in the interior of the nef cone, which is spanned, as observed before, by $H+E$ and $4H-E$. Hence, by Kodaira vanishing, we have $H^1(N^*_{S/{\PP^3}}(E))=H^1(\OO_S(-6H+E))=0$ as desired.
\end{proof}

Besides the sextic surfaces containing a plane quartic, we have studied other components of the Noether-Lefschetz locus. We could not find any pair $(S,\cL)$ for which $\Omega_S^1$ is not $\cL$-liftable. Motivated by this, we pose the following question:

\begin{question}
\label{QUE:SURFINP3}
Is there any surface $S$ in $\PP^3$ with a big and semiample line bundle $\cL$ for which $\Omega_S^1$ is not $\cL$-liftable?
\end{question}




\section{Holomorphic one forms on $\Mgo$}

Let $\cM_{g}$ be the coarse moduli space of smooth complex  projective curves of genus $g\geq 2$. We will use some classical results about $\cM_g$ and its compactifications. The reader can refer to \cite{ACG2} and \cite{HM}. The variety $\cM_{g}$ is quasi-projective and it is singular on points parametrizing curves with non-trivial automorphism group for $g\geq4$. We denote $\Mgo\subset \cM_g$ the locus of points parametrizing curves with trivial automorphism group, which coincides with the smooth locus for $g\geq 4$. Furthermore, the singular locus $\cM_g^{sing}=\cM_g\setminus \Mgo$ has codimension $g-2,$ since its largest subvariety is the hyperelliptic locus.
\\

We are interested in studying the cotangent bundle $\Omega^1_{\Mgo}$. Our result is the following:
\begin{theorem}
\label{THM:NOONEFORMS}
For $g\geq 5$, $\Mgo$ has no holomorphic forms, that is $H^0(\Omega^1_{\Mgo})=0$.
\end{theorem}

It is well known (see {\cite{Mu}) that $\Mgo$ has no closed holomorphic $1$-forms with respect to the de-Rham differential. For completeness, we will briefly recall this in Theorem \ref{THM:BASICCOH}. Nevertheless, $H^0(\Omega^1_{\Mgo})$ could still be non zero since $\Mgo$ is not compact and so holomorphic forms are not automatically closed.
\\


To prove the result, we need to use two  classical compactifications of $\cM_g$ that we recall now. The first one is the Deligne-Mumford compactification $\MDM$ (see \cite{HM}, \cite{ACG2}), which is defined as the coarse moduli space of stable curves of genus $g$. It is a projective variety and the boundary $\partial \MDM=\MDM\setminus \cM_g$ is $$\Delta=\Delta_0\cup \Delta_1\cup \Delta_2 \cup \cdots \cup \Delta_{[g/2]}.$$
It is build up as union of divisors $\Delta_i$, for $i=0,1, \dots, [g/2]$, characterized as follows. 
The generic point of $\Delta_0$ represents the class of an irreducible nodal curve with a single node and arithmetic genus $g$. The generic point of $\Delta_i$ with $i\geq 1$ represents the class of a nodal curve with two smooth components of genus $i$ and $g-i$, respectively, meeting at a single node. 
Let $\lambda$ be the Hodge class and recall that $\Delta_1$ is divisible by $2$ in the Picard group of $\MDM$. 
 
The canonical divisor of $\MDM$ can be written (see \cite[Page 160, Equation (3.113)]{HM}) as 
$$K_{\MDM}=13\lambda-\frac{3}{2}\Delta_1-2\Delta_0-2\sum_{i=2}^{[g/2]}\Delta_i=13\lambda-2\Delta+\frac{1}{2}\Delta_1.$$

\begin{remark}
\label{REM:DIVISIBILITY}
Recall that the canonical divisor $K_{\MDM}$ of the coarse moduli space is related to the canonical divisor $K$ of the moduli stack by the formula $\pi^\ast K_{\MDM}= K+\delta_1$ where $\pi$ denotes the projection from the moduli stack to the coarse moduli space and $\delta_1$ is the class of the divisor $\Delta_1$ in the moduli stack.
For a general
point $[C]\in \Delta_1$ we have in fact that the versal deformation space of $C$ is a two-sheeted cover onto image in $\MDM$, ramified over $\Delta_1$. So in the language of stack the morphism $\pi$ is ramified over $\Delta_1$ and therefore $\Delta_1$ is $2$-divisible. For a complete explaination see \cite[Page 160]{HM}.
\end{remark}

The second compactification we are interested in is the Satake compactification $\MSat$. It is constructed by considering the Satake compactification $\ASat$ of $\cA_g$ (see \cite{SAT,BB,Igu}) and the morphism $\Sat:\MDM\to \ASat$, defined set theoretically by sending a stable curve to the Jacobian of its normalization. Then $\MSat$ is defined as the image of $\Sat$ and it is projective since $\ASat$ is projective (see \cite{BB,Igu}). Let $\HSat$ be the ample class on $\ASat$ giving the Satake embedding and consider its pullback $H$ on $\MDM$, which is therefore big and semiample since the Torelli morphism is given by a multiple of $\lambda$.


Furthermore, $\Sat$ is a birational morphism that is injective on $\cM_g$ by the Torelli Theorem and contracts the divisor $\Delta_1$ to a subvariety of $\MSat$ of codimension $2$. Whereas, if $i\neq 1$, $\Delta_i$ is contracted to a subvariety having codimension $3$. 

 In other words, the morphism induced by a suitable multiple of $H$ factors as in the following diagram:

$$
\xymatrix@C=2cm@R=1cm{
\MDM \ar@{->>}[rd]_{\Sat} \ar[r]_{\Sat} \ar@/^2.0pc/@{^->}[rr]^{\varphi_{|dH|}} & 
\ASat \ar@{^{(}->}[r]_-{\varphi_{|d\HSat|}} & \PP^{N} \\
& \MSat \ar@{^{(}->}[u]
}
$$

We give the following definition:
\begin{definition}\label{DEF:SATAKEVAR} Let $\varphi_{|dH|}:\MDM\to\PP^{N}$ be the map introduced above and let $L\subset \PP^{N}$ be a linear subspace of codimension $c$.
We set $X_L= \varphi_{|dH|}^{-1}(L)$ and say that $X_L$ is an $H$-variety if $\dim X_L=3g-3-c$. In particular,  for $c=3g-5,$ $X_L$  is an $H$-surface and for $c=3g-4$ it is an $H$-curve.
\end{definition}







\subsection{Closed holomorphic forms on $\Mgo$} 

~\\ 
We denote by  $\Omega^i_{\Mgo,c}$ the kernel of $d: \Omega^i_{\Mgo}\to \Omega^{i+1}_{\Mgo}$ where $d$ is the holomorphic de Rham differential. Hence $\Omega^i_{\Mgo,c}$ is just the sheaf of $d$-closed holomorphic $1$-forms on $\Mgo$. 
%


\begin{theorem}
\label{THM:BASICCOH}
If $g\geq 4$, then $H^0( \cO_{\Mgo})=\bC$, $H^1(\Mgo, \bC)=0$ and $H^0(\Omega^1_{M^o_g,c})=0$.
\end{theorem}
\begin{proof} We first prove $H^0( \cO_{\Mgo})=\bC$. Consider a point $p\in \Mgo$, then we can cut out a smooth projective curve $C_q$ in $\Mgo$ passing through it and a general point $q\in \cM^o_g$ by using hyperplanes of $\Mgo\subset \MDM$. This can be done since we are assuming $g\geq 4$ and then the complement of $\Mgo$ inside $\MDM$ has codimension at least two. Then, for any $f\in H^0(\cO_{\Mgo})$, since $C_q$ is projective, $f|_{C_q}$ is constant. Then we have $f(p)=f(q)$, so $f$ is constant on $\Mgo$.

Let $\mathcal T_g$ be the Teichmuller space of Riemann surfaces of genus $g$ and consider the mapping class group $\Gamma_g$. The proof of $H^1(\Mgo,\bC)=0$ relies on the result about the abelianization of the mapping class group $\Gamma_g$: it is trivial as soon as $g\geq 3$ (see \cite{Mu,Har}). We recall that $\mathcal T_g$ is contractible and that $\Gamma_g$ acts properly discontinuously on $\mathcal T_g$  with quotient  $\cM_g.$ Let
$\pi: \mathcal T_g \to \mathcal T_g/\Gamma_g=\cM_g$ be the quotient map. Set $\mathcal T_g^o=\pi^{-1}(\Mgo),$ and let
$\pi^o: \mathcal T_g^o\to \Mgo$ be the restriction of $\pi$. The action of $\Gamma_g$ on $ \Mgo$
is free, and, for $g\geq 4$, $ \mathcal T_g\setminus \mathcal T_g^o$ has codimension in $ \mathcal T_g$ equal to $g-2\geq 2$. 
Therefore the fundamental groups of $\mathcal T_g^o$ and of $\mathcal T_g$ are isomorphic. It follows that $\pi^o$ is the universal covering of $\Mgo$. Then $\Pi_1(\Mgo,x_0)\simeq \Gamma_g$ and $H_1(\Mgo,\bC)\simeq H^1(\Mgo,\bC)=0$. 
	
Let $\eta$ be be a closed  holomorphic form on $\Mgo$. Then, by the above result, $\eta$ is also exact, that is there exists $f$ such that $\eta=df$. Since $\eta$ is holomorphic, $f$ is a holomorphic function and thus it is constant. Consequently, $\eta=0$.

\end{proof}
\begin{corollary} 
\label{COR:COHRES}
Let $\nu: M'\to \MDM$ be a resolution of singularities such that $\nu^{-1}(\Mgo)\simeq \Mgo$. Then $H^0(\Omega^1_{M'})=0$, $H^1(\cO_{M'})=0$ and $H^1(M', \bC)=0$.
	\end{corollary}
\begin{proof} Let  $\eta \in H^0(\Omega^1_{M'})$.  Since $M'$ is projective, then $\eta$ is $d$-closed and so its restriction $\eta|_{\nu^{-1}(\Mgo)}=\eta|_{\Mgo}$ to $\Mgo$ is $d$-closed, that is $\eta|_{\Mgo}\in H^0(\Omega^1_{M^o_g,c})$. By Theorem \ref{THM:BASICCOH}, $\eta|_{\Mgo}=0$ and then $\eta=0$. Since $M'$ is smooth we have also $H^1(\OO_{M'})=H^0(\Omega_{M'}^1)=0$.
\end{proof}


\subsection{$H$-surfaces}
\label{SUBS:HSURF}

~\\ 
From now on $S$ will be a general $H$-surface in $\MDM$, i.e. it is a general complete intersection of $3g-5$ hypersurfaces whose classes are suitable multiples of $H$. This surface $S$ will play a central role in the proof of Theorem \ref{THM:NOONEFORMS}. We now give a list of properties of $S$.

\begin{proposition}
\label{PROP:PROPOFS}
For $g\geq 5$, a general $H$-surface $S$ satisfies the following properties:
\begin{enumerate}[(a)]
\item\label{PROP:PROPOFSa}
$S$ is smooth and contained in the open set of smooth points of $\MDM$. Moreover, $\cO_S(n\lambda)\simeq \cO_S(H)$ for suitable $n>0$, where $\lambda$ denotes the restriction of the Hodge class to $S$.
\item\label{PROP:PROPOFSb}
We have $S\cap \Delta_i=\emptyset$ for $i\neq 1$ and $E=S\cap \Delta_1$ is an effective divisor which is disjoint union of smooth curves of genus $g-1$.
\item\label{PROP:PROPOFSc}
The canonical sheaf of $S$ is $\omega_S=\OO_S\left(k\lambda+\frac{3}{2}E\right)$ for some suitable $k>0$.
\item\label{PROP:PROPOFSd}
We have $H^0(\Omega_S^1)=0$.
\item\label{PROP:PROPOFSe}
The morphism $\Sat|_S$ is birational, contracts $E$ to a finite number of points, is an isomorphism outside $E$
and $H\cdot E=0$.
\item\label{PROP:PROPOFSf}
Fix a general point $p\in \Mgo$ and a general vector $v\in T_p\Mgo$. Then there exists an $H$-surface $S$ such that $v\in T_pS$.
\end{enumerate}
In particular, $S$ is a smooth regular surface in $(\Mgo\cup\Delta_1)\setminus \Sing(\MDM)$.
\end{proposition}

\begin{proof}
We recall that $\dim \Sat (\Delta_i)=3g-6$ for $i\neq 1$ and $\dim \Sat (\Delta_1)=3g-5$. The general point of $\Delta_1$ is a curve $B$ with one node and two smooth components given by an elliptic curve $D$ and a smooth curve $C$ of genus $g-1\geq 4$ that we can take without non-trivial automorphisms. Then $\Sat(B)=D\times JC$ and the fiber of $\Sat$ over this point, by Torelli's Theorem, is described as the curves obtained by glueing $D$ and $C$ at a point. Because of the translations on $D$, the fiber over $D\times JC$ has dimension $1$ and can be identified by $C.$  


Consider the locus $\Delta_{1,s}$ of singular points of $\MDM$ lying in $\Delta_1$, i.e. $\Delta_{1,s}=\Sing(\MDM)\cap \Delta_1$, and set 
$$Y=\Sat(\Delta_0\cup  \Delta_{1,s}\cup \Delta_{2}\cup\dots \cup\Delta_{[g/2]}\cup \Sing(\cM_g)).$$
We claim now that, under the assumption $g\geq 5$,  $Y$ has codimension $3$ in $\MSat$.
It is enough to show that the image of $\tilde{Y}=\left(\Delta_{1,s}\setminus \bigcup_{i\neq 1} \Delta_i\right)$ has codimension $3$. Moreover, we can restrict to the subspace of $\tilde{Y}$ represented by curves with at most $2$ nodes (since curves with more than $2$ nodes define loci of dimension $3g-6$). If $B$ is one of such a curves we have three possible cases: 
\begin{description}
    \item [(1)] $B$ is a curve which is the union of a smooth curve $C$ of genus $g-1$ and an elliptic curve meeting at a point $P$. We distinguish two subcases: 
    \begin{description}
        \item [(1.a)] $\Stab_{\Aut(C)}(P)\neq \{\id\}$;
        \item [(1.b)] $\Stab_{\Aut(D)}(P)\neq \{\pm\id\}$.
    \end{description}
    \item [(2)] $B$ is a curve which is the union of a smooth curve $C$ of genus $g-2$ and $2$ disjoint elliptic curves $E_1,E_2$ such that $E_i\cap C=Q_i$ is a point. 
\end{description}
We denote by $\tilde{Y}_{(1.a)}$, $\tilde{Y}_{(1.b)}$ and $\tilde{Y}_{(2)}$ the corresponding loci in $\MDM$. First of all, notice that $\Sat(\tilde{Y}_{(2)})$ has dimension $3g-7$. The locus $\tilde{Y}_{(1.a)}$ has dimension at most $2g-2$ (which is the dimension of the locus of hyperelliptic curves of genus $g-1$ plus the dimension of the moduli of elliptic curves) and so $\Sat(\tilde{Y}_{(1.a)})$ has codimension more than $3$.
Finally, notice that $\tilde{Y}_{(1.b)}$ has dimension $3g-5$ but it is also true that all the fibers of $\Sat$ of points of $\tilde{Y}_{(1.b)}$ is contained in $\tilde{Y}_{(1.b)}$, so $\Sat(\tilde{Y}_{(1.b)})$ has dimension $3g-6$.



(a) To prove that a general $S$ is smooth, we can assume first $\Sat(S)\cap Y=\emptyset$.  It follows in particular that $S$ is disjoint from $\Sing(\MDM)$. Moreover, since $H$ is semiample on $\MDM$, the general $S$ does not have singularities by Bertini's Theorem. In order to see that $\cO_S(H)$ is a positive multiple of $\cO_S(\lambda)$ it is enough to observe that $S$ is disjoint from $\Delta_0$. Indeed, this imply that the closure of $\Sat(S\setminus \Delta_1)$ in $\ASat$ is contained in $\cA_g$. Then, the claim follows since the Picard group of $\cA_g$ is spanned by the Hodge class $\lambda_{\cA_g}$ (whose pullback is $\lambda$, by definition).
\\
(b) The argument above shows that  for $S$ general,  $S\cap \Delta_i=\emptyset$ for $i\neq1$.
As $\HSat$ is ample we have that the image $\Sat(S)$ needs to intersect $\Sat(\Delta_1)$ in a finite number of points 
$q_1,\dots,q_r.$  Hence, $S$ intersects $\Delta_1$ in a divisor, which we will denote by $E$, whose components are a finite number of disjoint smooth curves of genus $g-1$ (again one uses Bertini Theorem for the restriction $\Sat|_{\Delta_1}$). Then, by construction, the divisor $E$ is two divisible, since $\Delta_1$ is two divisible as observed in Remark \ref{REM:DIVISIBILITY}.

(c) As $S$ is a complete intersection in the nonsingular locus of $\MDM$ and since $\cO_S(H)=\cO_S(n\lambda)$ with suitable $n>0$ (as proved in $(a)$), by adjunction we have
$$\omega_S=\OO_S\left(13\lambda-\frac{3}{2}\Delta_1-2\sum_{i\neq 1}^{[g/2]}\Delta_i+m H\right)=\OO_S\left(k\lambda-\frac{3}{2}E\right)$$
where $m>0$ and $k>0$.

(d) 
Let $\nu:M'\to \MDM$ be a desingularization which induces an isomorphism $\nu^{-1}(\Mgo)\simeq \Mgo$. Then, by Corollary \ref{COR:COHRES}, we have that $H^1(\OO_{M'})=0$. Let $H'$ be the pullback of $H$ with respect to $f$. As $S$ is a complete intersection in $(\Mgo\cup \Delta_1)\setminus \Sing(\MDM)$ we can realize $S$ as a complete intersection in $M'$ by using multiples of the big line bundle $H'.$  
The statement follows by this application of the Kawamata-Viewheg vanishing Theorem (\cite{Kaw}).

\begin{center}
\textit{Let $Z$ be a smooth variety of dimension $\dim(Z)\geq 3$ with $h^1(\OO_Z)=0$ and let $H$ be a big and nef divisor on $Z$. Then, the general element $Y\in |H|$ is smooth and is such that $h^1(\OO_Y)=0$.}
\end{center}

Indeed, starting from $M'$ and cutting with multiples of $H'$ to obtain $S$, we have $$h^0(\Omega_S^1)=h^1(\OO_S)=\cdots=h^1(\OO_{M'\cap m_1H'})=0.$$  

(e) The last statement follows immediately by the construction.

(f) Since $p$ and $v$ are generic (in $\Mgo$ and $T_p\Mgo$, respectively), the general $H$-surface has the desired property.

\end{proof}


\subsection{Concluding the proof of Theorem \ref{THM:NOONEFORMS}}
\label{SUBS:ENDPROOF}

~\\ 
The following result is the main technical tool.

\begin{proposition}
\label{PROP:CANAPPLYMAINTHM}
The sheaf $\Omega_S^1$ is $\OO_S(H)$-liftable.
\end{proposition}

The proof of this proposition follows directly by applying Theorem \ref{thm:doubled}, Proposition \ref{PROP:PROPOFS}(\ref{PROP:PROPOFSe}) and the following Lemma.

\begin{lemma}
\label{LEM:SURJSUP}
With the above notations, $H^0(\Omega_S^1)\simeq H^0(\Omega_S^1(mE))=0$,  for all $m\geq 0$.
\end{lemma}

\begin{proof}
By Proposition \ref{PROP:PROPOFS}(\ref{PROP:PROPOFSd}) we have $h^0(\Omega_S^1)=0$. So we only have to show that $H^0(\Omega_S^1)\simeq H^0(\Omega_S^1(mE))$, for all $m$. We perform this in two steps:  we prove first that $H^0(\Omega_S^1)\simeq H^0(\Omega_S^1(E))$ and then that $H^0(\Omega_S^1(mE))\simeq H^0(\Omega_S^1((m+1)E))$ for $m\geq 1$.
\\
	
We start with the first claim. Consider the exact sequences
$$(I):\qquad 0\to \Omega_S^1\to \Omega_S^1(E)\to \Omega_S^1(E)|_E\to 0,$$
$$(II):\qquad 0\to \OO_E\to \Omega_S^1(E)|_E\to \omega_E(E)\to 0.$$
By Proposition \ref{PROP:PROPOFS}(\ref{PROP:PROPOFSd}) we have that $E=\sum_{i=1}^k E_i$ where $E_i$ are smooth disjoint curves of genus $g-1$. Since $\Sat|_S$ contracts $E_i$ we have $H|_S\cdot E_i=0$ for all $i$ and so $H|_S\cdot E=0$.
Then, by Proposition \ref{PROP:PROPOFS}(\ref{PROP:PROPOFSc}) and by adjunction we have
$$\omega_E=\omega_S\otimes \OO_E(E)=\OO_E\left(k\lambda-\frac{3}{2}E+E\right)=\OO_E\left(-\frac{1}{2}E\right).$$

Since $E_i$ is effective and $H|_S\cdot E_i=\lambda|_S\cdot E_i=0$, by the Hodge index theorem, we have $E_i^2<0$ and so $H^0\left(\cO_{E_i}\left(\frac{1}{2}E_i\right)\right)=0$. Being $E_i$ and $E_j$ disjoint we have $E_i\cdot E_j=0$ and $\cO_E=\bigoplus_i \cO_{E_i}$. In particular, $\omega_E(E)=\bigoplus_{i}\omega_{E_i}(E_i)=\bigoplus_{i}\cO_{E_i}\left(\frac{1}{2}E_i\right)$, so $H^0(\omega_E(E))=0$.

Using this and the exact sequence $(II)$ we have that
$$H^0(\OO_E)\xrightarrow{\alpha} H^0(\Omega_S^1(E)|_E)$$
is an isomorphism.  Consider the sequence
$$(III):\qquad 0\to \Omega_S^1\to \Omega_S^1(\log(E))\stackrel{res}\la \OO_E\to 0.$$
where $\Omega_S^1(\log (E))$ is the bundle of logarithmic differentials with poles along $E$. We have the following commutative diagram given by $(I),(II)$ and $(III)$.
$$
\xymatrix{
0 \ar[r] &
    H^0(\Omega_S^1) \ar@{^{(}->}[r]^{\iota}  \ar[d]_{=}&
    H^0(\Omega_S^1(\log E)) \ar[r] \ar@{->>}[d]& 
    H^0(\OO_E) \ar[r]^-{\partial} \ar@{->}[d]_{\alpha}^-{\simeq}&
    H^1(\Omega_S^1)\ar[d]_{=}\\
0 \ar[r] &
    H^0(\Omega_S^1) \ar@{^{(}->}[r] & 
    H^0(\Omega_S^1(E)) \ar[r] & 
    H^0(\Omega_S^1(E)|_E) \ar[r]^-{\partial'} &
    H^1(\Omega_S^1)
}
$$

Then we obtain $\partial'=\partial \circ \alpha.$ One has to show that $\partial' $ is injective or equivalently that $\partial$ is injective.
As observed above, we have $H^0(\OO_E)=\oplus_i^k H^0(\OO_{E_i})$
and we can write
$$\partial: \bigoplus_i^k H^0(\OO_{E_i})\to H^1(\Omega_S^1).$$ 
By (see \cite[Pag. 458-459]{GH})  this is just obtained by the Atiyah-Chern class (\cite{Ati}) via residues computation
$$(a_1,\dots,a_k)\mapsto \sum a_i c_1(E_i).$$  
The $E_i$ are effective and disjoint divisors with negative self-intersection, as previously observed. It follows that their first Chern classes are independent and hence that $\partial$ is injective.
In conclusion  $\iota:H^0(\Omega_S^1)\hookrightarrow H^0(\Omega_S^1(E))$
given by $(I)$ is an isomorphism and $0=h^0(\Omega_S^1)=h^0(\Omega_S^1(E))$.

We now prove the second part: for $m\geq 1$,   $H^0(\Omega_S^1(mE))\simeq H^0(\Omega_S^1((m+1)E))$. 
Consider the exact sequences
$$(I)'\qquad 0\to \Omega_S^1(mE)\to \Omega_S^1((m+1)E)\to \Omega_S^1((m+1)E)|_E\to 0$$
$$(II)'\qquad 0\to \OO_E(mE)\to \Omega_S^1((m+1)E)|_E\to \omega_E((m+1)E)\to 0$$
obtained from $(I)$ and $(II)$ respectively. We have seen that
$\omega_E=\OO_E\left(-\frac{1}{2}E\right)$, so both $\omega_E((m+1)E)=\OO_E\left(\left(m+\frac{1}{2}\right)E\right)$ and $\OO_E(mE)$ have negative degree for $m\geq 1.$ Then, $H^0(\Omega_S^1((m+1)E)|_E)=0$. This yields the desired result from the exact sequence $(I)'$ and induction.
\end{proof}

We can now conclude the proof of Theorem \ref{THM:NOONEFORMS}.

\begin{proof}[Proof of Theorem \ref{THM:NOONEFORMS}]
Let $\eta \in H^0(\Omega^1_{\Mgo})$. If $\eta\not\equiv 0$, then for a general point $p$ in $\Mgo$, $\eta_p\in \Omega^1_{\Mgo, p}$ is not identically zero. Let $v$ be a general element in $T_p\Mgo$. Hence we can assume $\eta_p(v)\neq 0$. By \ref{PROP:PROPOFS}(\ref{PROP:PROPOFSf}) we can find a general $H$-surface $S$ which passes through $p$ and such that $v\in T_pS$. Consider the open subset $U=S\setminus E$. Recall that $\cO_S(H)$ is big and semiample and $E\cdot H|_S=0$, so $U$ is an open neighborhood of a general curve in $|dH|_S$. By construction, the restrition $\eta_U$ of $\eta$ to $U$ defines a nontrivial element of $H^0(\Omega_S^1|_U)$.  But now, by Proposition \ref{PROP:CANAPPLYMAINTHM} we have that $\Omega_S^1$ is $\cO_S(H)$-liftable. Then, by Theorem \ref{THM:EXTENSIONFROMOPENS} we can conclude that $\Omega_S^1$ is $\cO_S(H)$-concave so $H^0(\Omega_S^1|_U)\simeq H^0(\Omega_S^1)\neq 0$. On the other hand, this yields a contradiction since, by Proposition \ref{PROP:PROPOFS}(\ref{PROP:PROPOFSd}), $H^0(\Omega_S^1|_U)=0$. \end{proof}

\begin{remark}
Notice that the assumption $g\geq 5$ is necessary in order to have $S$ smooth. Indeed, if $g=4$, the general $S$ meets the hyperelliptic locus (which has codimension $2$ in the moduli space) in a finite number of points so the general $S$ has a finite number of nodes as singularities. Nevertheless, the theorem should follow just by blowing up the points. When $g=2$, $\Mgo=\emptyset$ and when
$g=3$ one has to remove the hyperelliptic divisor  that is ample. The vanishing of all holomorphic forms on the open set of the smooth locus could still hold. Some analysis of the singularities and  of the fixed points for the action of the mapping class group is however necessary.
\end{remark}

We now gives a version of Theorem \ref{THM:BASICCOH} for certain open analytic subsets containing an $H$-curve, which are not necessarily open for the Zariski topology. In this sense this provides a strengthen of the theorem. 
\begin{theorem} \label{MgCON} 
Let $g\geq5$ and let $C$ be an $H$-curve in $\Mgo$. Let $U\subseteq \Mgo$ be a connected open neighborhood of $C$ for the classical topology. Then, we have $H^0(\Omega^1_U)=0$.
\end{theorem}

\begin{proof} Assume by contradiction that $H^0(\Omega^1_U)\neq 0$ and fix $\eta\in H^0(\Omega^1_U),$ $\eta\neq 0.$  Recall that $C$ is complete intersection of $3g-5$ general hypersurfaces in $|aH|$. More precisely we can find $L\in \Gr(|aH|,3g-5)$ general such that $\varphi_{|aH|}^{-1}(L)=C$.
Since being contained in $U$ gives an open condition in $\Gr(|aH|,3g-5)$ (in the analytical topology), by moving $L$ we can find $U'$ with $C\subset U'\subseteq U$ such that $U'$ is covered by $H-$ curves and the tangent vectors of those curves span the tangent space of $U'$ at the general point $p\in U'$. We can then find a general  
$H$-curve $C'\subset U'$ corresponding to $L'\in \Gr(|aH|,3g-5)$ such that $\eta_{C'}\neq 0$ where $\eta_{C'}\in H^0(\Omega^1_{C'})$ is the restriction of $\eta.$
Then, the general element of $\Gr(|aH|,3g-4)$ that contains $L'$ yields a smooth $H$-surface $S$ that contains $C'$. Let $U_S$ be the connected component of $U\cap S$ that contains $C'.$  One has that the restriction $\eta_{U_S}\in H^0(\Omega^1_{U_S})$ of $\eta$  is a fortiori non zero. But $\Omega^1_{S}$ is $H$-liftable and, by Theorem  \ref{THM:EXTENSIONFROMOPENS}, $H$-concave.
Since we have $C'\in |aH|$, by Proposition \ref{PROP:PROPOFS} we get $ 0=H^0(\Omega^1_S)\simeq H^0(\Omega^1_{U_S})$. This implies $\eta_{U_S}=0$, which gives a contradiction.
\end{proof}


\subsection{Holomorphic one-forms on moduli spaces of marked curves}
\label{SUBS:MARKED}

~\\ 
In this subsection we extend the result of Theorem \ref{THM:NOONEFORMS} to the moduli space of marked curves. We denote by $\cM_{g,n}$ the coarse moduli space of $n$-marked smooth projective curves of genus $g$. Denote by $\cM^o_{g,n}\subset \cM_{g,n}$ the smooth locus of $\cM_{g,n}$. We prove the following.

\begin{theorem}
\label{THM:NOONEFORMSMARKED}
Let $g\geq 5$. Then $\cM^o_{g,n}$ has no holomorphic $1$-forms for any $n\geq 0$, that is $H^0(\Omega_{\cM^o_{g,n}}^1)=0$, for any $n\geq 0$.
\end{theorem}
\begin{proof}
Consider the morphism $f^n: \cM_{g,n} \to \cM_{g,n-1}$ that  forgets the last marked point, i.e. 
$$f^n:[C, p_1, \dots ,p_{n-1}, p_n]\mapsto [C, p_1, \dots ,p_{n-1}].$$  
Set $\Ug{0}=\cM^o_g$ and
$\Ug{n}=(f^{n})^{-1}(\Ug{n-1})$, for any $n\geq 1$. Note that $\Ug{n}\subset \cM^o_{g,n}$ is a Zariski open set parametrizing marked curves, whose underlying curve has trivial automorphism group. We now prove that $H^0(\Omega_{\Ug{n}}^1)=0$, for any $n\geq 0$. From this we conclude that $H^0(\Omega^1_{\cM^o_{g,n}})=0$. In fact, $\Ug{n}\subset \cM^o_{g,n}$ is an open subset, so any non-vanishing form on $\cM^o_{g,n}$ would restrict to the zero form on an open set $\Ug{n}$ which is not possible.
\smallskip

The proof is by induction on $n$. The case $n=0$ is the content of Theorem \ref{THM:NOONEFORMS}. We now assume that $H^0(\Omega^1_{\Ug{n-1}})=0$ and we prove that $H^0(\Omega^1_{\Ug{n}})=0$.

For simplicity we define $f$ to be the restriction of $f^n$ to $\Ug{n}$. Then, $f: \Ug{n}\to \Ug{n-1}$ is a family of smooth projective curves of genus $g$ over a smooth variety of dimension $3g-3+(n-1)$. Indeed, the fibre $f^{-1}([C, p_1, \dots ,p_{n-1}])$ is naturally isomorphic to the curve $C$ because $C$ has trivial automorphism group. Consider the short exact sequence of relative differentials
$$
\xymatrix{
0\ar[r] &
    f^\ast \Omega^1_{\Ug{n-1}} \ar[r]^-{df} &
    \Omega^1_{\Ug{n}} \ar[r] &
    \Omega^1_{\Ug{n}/\Ug{n-1}}\ar[r] &
    0.
}$$
Since $f$ is a fibration the push-forward and the projection formula give the exact sequence 	
$$
\xymatrix{
0\ar[r] &
    \Omega^1_{\Ug{n-1}} \ar[r]^-{df} &
    f_{\ast}\Omega^1_{\Ug{n}} \ar[r] &
    f_{\ast}\Omega^1_{\Ug{n}/\Ug{n-1}}\ar[r]^-{\partial} &
    R^1f_{\ast}\OO_{\Ug{n}}\otimes \Omega^1_{\Ug{n-1}}.
}$$
We now show that $\ker(\partial)=0$. Then, by induction we get $0=H^0(\Omega^1_{\Ug{n-1}})\simeq H^0(f_\ast\Omega^1_{\Ug{n}})\simeq H^0(\Omega^1_{\Ug{n}})$ for all $n\geq 1$, which ends the proof. 

Fix $x=[C,p_1,\dots,p_{n-1},p_n]\in \Ug{n}$. We can describe $\partial_x$ as the homomorphism
$$\partial_x: H^0(\omega_C)\to H^1(\OO_C)\otimes (\Omega_{\Ug{n-1}}^1)_x=\Hom (H^0(\omega_C),(\Omega_{\Ug{n-1}}^1)_x)$$
since $f_{\ast}\Omega^1_{\Ug{n}/\Ug{n-1}}$ is the Hodge bundle and $R^1f_{\ast}\OO_{\Ug{n}}$ is its dual. 
The forgetful map $F:\Ug{n}\to \Mgo$ induce the sequence
$$
\xymatrix{
0\ar[r] &
    F^\ast\Omega_{\Mgo}^1 \ar[r]^-{dF} &
    \Omega_{\Ug{n}}^1 \ar[r] &
    \Omega_{\Ug{n}/\Mgo}^1 \ar[r] &
    0.
}$$
In $x$, we have $(\Omega_{\Mgo}^1)_x=H^0(\omega_C^{2})$ (see \cite[Chapter XI]{ACG2}). Then we have an inclusion $\iota:H^0(\omega_C^2)\to (\Omega_{\Ug{n-1}}^1)_x$ given by $dF_x\circ (F^*)_x$. Note that, by construction, points in the fiber of $f$ have the same image via $F$ so the Hodge bundle is constant along these fibers. Consider the multiplication map $\mu: H^0(\omega_C)^{\otimes 2}\to H^0(\omega_C^2)$ and define $\psi:H^0(\omega_C)\to \Hom (H^0(\omega_C),H^0(\omega_C^2))$ as $\psi(\alpha)=\mu(\alpha\otimes -)$. Then, by construction, $\partial_x=\iota\circ \psi$ so we have a commutative diagram
$$
\xymatrix@C=1cm{
H^0(\omega_C)\ar[rrr]^-{\psi} \ar[d]_{=}& & &
    H^1(\OO_C)\otimes H^0(\omega_C^{2}) \ar@{}[r]|-{=} \ar[d]^{\id\otimes \iota} &\Hom (H^0(\omega_C),H^0(\omega_C^2)) \ar[d]^{\iota \circ -}\\
H^0(\omega_C)\ar[rrr]^-{\partial_x} & & &
    H^1(\OO_C)\otimes (\Omega_{\Ug{n-1}}^1)_x  \ar@{}[r]|-{=} & \Hom (H^0(\omega_C),(\Omega_{\Ug{n-1}}^1)_x).
}
$$
In particular, since $\psi$ and $\iota$ are injective by construction, we have that $\partial_x$ is injective.
Then $\ker(\partial)=0$ as claimed.

\end{proof}

\end{document}